\documentclass[a4paper,11pt]{article}
\usepackage{geometry}
\usepackage{fullpage}
\usepackage[parfill]{parskip}    
\usepackage{graphicx}
\usepackage{amssymb}
\usepackage{amsmath}
\usepackage{epstopdf}
\usepackage{amsmath}

\usepackage{amsmath,amsthm,amscd,amssymb}
\usepackage{latexsym}
\usepackage[colorlinks,citecolor=red,pagebackref,hypertexnames=false]{hyperref}

\numberwithin{equation}{section}

\theoremstyle{plain}
\newtheorem{theorem}{Theorem}[section]
\newtheorem{Def}[theorem]{Definition}
\newtheorem{lemma}[theorem]{Lemma}
\newtheorem{corollary}[theorem]{Corollary}
\newtheorem{proposition}[theorem]{Proposition}

\theoremstyle{definition}
\newtheorem{definition}[theorem]{Definition}

\theoremstyle{remark}

\newtheorem{case[theorem]}{Case}

\def\supp{\hbox{supp\,}}
\def\norm#1.#2.{\lVert#1\rVert_{#2}}

\title{Tight Wavelet Frame Sets in Finite  Vector Spaces}

\author{Alex Iosevich,
Chun-Kit Lai, Azita Mayeli}

\begin{document}
\maketitle

\begin{abstract} Let $q\geq 2$ be an integer,  and $\Bbb F_q^d$, $d\geq 1$,  be the vector space over the cyclic space   $\Bbb F_q$. The purpose of this paper is two-fold. First, we obtain sufficient conditions on $E \subset \Bbb F_q^d$ such that the inverse Fourier transform of $1_E$ generates  a tight wavelet frame  in $L^2(\Bbb F_q^d)$. We call these sets (tight) wavelet frame sets. The conditions are given in terms of multiplicative and translational tilings, which is analogous with Theorem 1.1 (\cite{Wang}) by Wang in the setting of finite fields. In the second part of the paper, we exhibit a constructive method for obtaining  tight wavelet frame sets in $\Bbb F_q^d$, $d\geq 2$, $q$ an odd prime and $q\equiv 3$ (mod 4). \end{abstract}


\section{Introduction}
A countable subset $\{x_k\}_{k\in I}$ of a Hilbert space ${\mathcal H}$ is said to be a {\it frame}    if there exists two positive  constants $A\leq B$ such that for any $x\in \mathcal H$
$$
A\|x\|_{\mathcal H}^2 \leq \sum_{k\in I} |\langle x, x_k\rangle_{\mathcal H}|^2\leq B \|x\|_{\mathcal H}^2 .
$$
The positive constants $A$ and $B$ are called lower and upper frame bounds, respectively. The frame is called a {\it tight frame} when we can take $A=B$ and it is called a {\it Parseval frame} if $A=B=1$. One of most significant features of the frames that makes them practical and useful is their  redundancy which has an important role, for example,  in robustness. The frames also allow a localized representation of elements in the Hilbert space and they have   been used for a number of years by engineers and applied mathematicians for purposes of signal processing and data compression.  The notion of frames  was first introduced by Duffin and Schaeffer \cite{Duffin-Schaeffer}.   Amongst the frames,   tight frames play a fundamental role in the applications of frames due to their numerical stability. In this paper we aim to construct tight frames on the  finite vector spaces over the finite fields which arise from dilation and translation of a function whose Fourier transform is characteristic function of a non-empty set.

\medskip

In the classical setting, a function $\psi\in L^2(\Bbb R^d)$ is said to generate a {\it orthonormal wavelet basis} (resp. {\it wavelet frame}) if there is a set of $d\times d$ matrices $\mathcal D\subset GL(d,\Bbb R)$ and a subset $T\subset \Bbb R^d$ such that the family
 \begin{align}\label{wavelet system}
\{|\det(D)|^{1/2} \psi(Dx -t):  \ D\in\mathcal D, t\in T\},
 \end{align}
forms an orthonormal basis  (resp. frame)  for $L^2(\Bbb R^d)$. Then we say  $\psi$ is an  {\it orthonormal wavelet (resp. frame wavelet)}   and every element $|\det(D)|^{1/2} \psi(Dx-t)$ is a  dilation and translation copy of $\psi$ with respect to the invertible matrix $D$ and translation $t$, respectively.
The structures of $\mathcal D $ and $T$ and associated to which there exists  orthonormal and frame wavelets for $L^2(\Bbb R^d)$  have been studied by many authors, for example \cite{Han-Larson, Larson-Speegle-Dai, Fang-Wang,Wang}. See also \cite{ILMP16} for an alternative perspective on wavelets in vector spaces over finite fields and connections with combinatorial problems. 

\medskip

Let $q$ be an odd prime  and ${\mathbb F}_q$ be the prime field with $q$ elements. Then ${\Bbb F}_q^d$ is the vector space of dimension $d$ over the finite field ${\Bbb F}_q$. In analogue to ${\mathbb R}^d$, the purpose of this paper is to study tight wavelet frames  on  $\mathbb F_q^d$ and its subspace for  $d\geq 1$. Let Aut$({\Bbb F}_q^d)$ be the set of all automorphisms on ${\mathbb F}_q^d$.
Following the same spirit of the frame wavelets on $\Bbb R^d$,  we say that a function $\psi: \ \mathbb F_q^d \to \Bbb C$ is a {\it wavelet} or a {\it  frame wavelet} for $L^2({\mathbb F}_q^d)$ if there exists a set of automorphisms ${\mathcal A}\subset {\text Aut}({\Bbb F}_q^d)$ and a subset ${\Lambda}\subset\mathbb F_q^d$ such that
the family
 \begin{align}\label{discrete wavelet system}
 \{ \psi(a x - \lambda): \ a\in \mathcal A,  \ \lambda\in \Lambda \}.
 \end{align}
is an orthonormal basis (resp. frame) for $L^2({\mathbb F}_q^d)$.  Note that in the continuous case, for the matrix $D$, the factor $\Delta_D:=|\det(D)|^{1/2}$ makes the dilation map 
$$f\to |\det(D)|^{1/2} f(Dx)$$ an isometry. In the discrete  case $\Bbb F_q^d$ the Haar measure on $\Bbb F_q^d$ is discrete and invariant under the dilation, therefore the  dilation factor $\Delta_a$ is $1$.

\medskip

A common way to construct a wavelet frame on ${\mathbb R}^d$ is to choose a  function whose Fourier transform is the indicator   of a measurable set, and then  consider the system (\ref{wavelet system}) of dilations and translations of the function.  This leads to the traditional definition of  frame wavelet sets as follows:
A set $\Omega\subset \Bbb R^d$ is called a {\it frame wavelet set} with respect to   $\mathcal D$  and $T$ if for the function $\psi$ with $\hat\psi= 1_\Omega$, the system  (\ref{wavelet system}) is a frame  in $L^2(\Bbb R^d)$.  If the system is an orthonormal basis for $L^2(\Bbb R^d)$,  then the  function $\psi$ is called minimally supported frequency wavelet  (MSF wavelet) and $\Omega$ is called a {\it wavelet set}.  The wavelet sets and minimally supported frequency wavelets   were introduced in \cite{Fang-Wang} and   studied exclusively,  for example, in  \cite{HWW1, HWW2} and by many other authors. The existence of wavelet sets  in $\Bbb R^d$ for any expansive matrix  was given in \cite{Larson-Speegle-Dai}.  The wavelet theory and wavelet sets are  studied in a constructive way in locally compact abelian groups with compact open subgroups, such as $p$-adic groups $\Bbb Q_p$,  in \cite{Benedetto-Benedetto04,Benedetto-Benedetto11}

\medskip

 A well-known example of a wavelet set is the {\it Shannon set} given by
 $$\Omega= [-2\pi, -\pi] \cup [\pi, 2\pi). $$
 The orthonormal wavelet for $\Omega$ is then given by $\psi(x)= 2\ \text{sinc}(2x-1) - \text {sinc}(x)$ with $\widehat{\psi} = 1_{\Omega}$, where $1_\Omega$ is the indicator   of  set $\Omega$ .  For more examples and constructions of wavelet sets in $\Bbb R^d$ we invite the reader to see \cite{Baggett99,Benedetto-Benedetto11,Benedetto-Leon99,Merrill08,Merrill12}.

 \medskip

 In \cite{Wang}, Wang tied the existence of wavelet sets with the notion of spectral sets. We say that a measurable set of positive measure $\Omega$ is a {\it spectral set} if there exists a countable set $\Gamma$ such that the collection of exponentials $\{e^{2\pi i \gamma\cdot x}: \gamma\in \Gamma\}$ forms an orthonormal basis for $L^2(\Omega)$.  In this case we  say that $(\Omega, \Gamma)$ forms a {\it spectral pair}. Spectral sets were first introduced by Fuglede \cite{F} and he proposed an infamous conjecture asserting that spectral sets are exactly translational tiles on ${\mathbb R}^d$. However, this conjecture was proven to be false in its full generality by Tao \cite{T}. Nowadays, the exact relationship between spectral sets and translational tiles are mysterious. We   refer to \cite{DL, IM_FA} for some recent progress.

  Wang proved the following result in the classical setting which   characterizes the wavelet sets  by  multiplicative and translational tiling. Let $D^T$ denote the transpose of matrix $D$, and  $\check g$  denote the  inverse Fourier transform of $g$.
 \medskip
    \begin{theorem}[Theorem 1.1, \cite{Wang}]\label{Wang}
    Let $\mathcal D\subset GL(d,\Bbb R)$ and  $\Gamma\subset \Bbb R^d$. Let $\Omega\subset \Bbb R^d$ with positive and finite Lebesgue measure.  If $\{D^t(\Omega): \ D\in \mathcal D\}$ is a tiling of $\Bbb R^d$ and $(\Omega, \Gamma)$ is a spectral pair, then $\psi= \check{1_\Omega}$ is a wavelet with respect to the dilation set $\mathcal D$ and  the translation set $\Gamma$.  Conversely, if $\psi= \check{1_\Omega}$ is a wavelet with respect to $\mathcal D$ and $\Gamma$ and $0\in \Gamma$, then $\{D^T(\Omega): D\in \mathcal D\}$ is a tiling of $\Bbb R^d$ and $(\Omega, \Gamma)$ is a spectral pair.
    \end{theorem}

    \medskip
Inspired by the result of the theorem,  it is natural for us to ask for what degree  one can     extend the notion  and concept of wavelet sets and multiplicative tiling  in  $\Bbb F_q^d$.   In what follows, we shall study this. For the multiplicative tiling purpose, we have to remove the origin and let  $Y:=\Bbb F_q^d\setminus \{0\}$.  

\medskip

\begin{Def}[Multiplicative and translational  tiling]
 {\rm Let $E$ be a subset of ${\Bbb F}_q^d$. We say   $E$ is a {\it multiplicative tiling set}  for  ${\Bbb F}_q^d$   if there is a set of automorphisms ${\mathcal A}$ in Aut$({\Bbb F}_q^d)$ such that $E$ tiles
$Y$
multiplicatively by   ${\mathcal A}$, i.e.,
$$
Y= \bigcup_{\alpha\in{\mathcal A}} \alpha(E) \ \mbox{(disjoint union)}.
$$
As a result, a multiplicative tiling set    does  not include the origin $\vec 0$. This is a natural requirement since $\alpha(\vec{0}) = \vec{0}$ for any automorphism $\alpha$.

\medskip

{\rm We say that a subset $F\subseteq {\Bbb F}_q^d$ is a {\it translational tiling set} for  ${\Bbb F}_q^d$ if there exists $\Lambda\subseteq {\Bbb F}_q^d$ such that }
$$
{\Bbb F}_q^d = \bigcup_{\lambda\in \Lambda} (F+\lambda) \ \mbox{(disjoint union)}.
$$
We say a set $E$ has a spectrum $L$ if the characters $\{\chi_l\}_{l\in L}$ is an orthonormal basis for $L^2(E)$. In this case, we say $E$ is a spectral and $(E, L)$ is a spectral pair.
}
\end{Def}

\medskip

      Our first result shows that  there exists {\it no} wavelet sets in the traditional sense for $L^2(\Bbb F_q^d)$.  More precisely,  there is no Parseval frame (thus no orthonormal basis) of type 
       (\ref{discrete wavelet system}) for  $L^2(\Bbb F_q^d)$ generated by any function of type $\psi:=  \check{1_E}$ (Theorem \ref{NO-ON-wavelet-basis}). However, later in this paper we   prove the existence of subsets $E$ in ${\mathbb F}_q^d$  for  which $\psi:= \check{1_E}$ generates  a  tight wavelet frame for {\it a subspace of $L^2(\Bbb F_q^d)$}. We shall call these sets {\it tight wavelet frame sets}.  We will then present an explicit construction  of  a class of  wavelet frame sets when $q\equiv 3$ (mod 4).
 %

We organize the paper as follows: In Section \ref{Sec:Parseval Wavelet Frames}   
 we prove an analogy of Theorem \ref{Wang} (Theorem 1.1, \cite{Wang}) in $\Bbb F_q^d$ and we 
study   necessary and sufficient conditions for a set $E$ such that $E$ is a tight wavelet frame set for a subspace of $L^2(\Bbb F_q^d)$. 
     These results are collected in Theorems \ref{PWF} and \ref{orthogonal system}. In this section we also provide a counter example where the disjointness of the sets does not necessarily hold  for tight wavelet frame sets.  In Section \ref{Existence of wavelet sets}, we will show the existence of a multiplicative tiling set  in ${\mathbb F}_q^d$ for $q$ prime and  $q\equiv 3$ (mod 4). 
 %
 %
 In Section \ref{existence of admissible sets}  we present  a constructive   approach in  Theorem \ref{Thm:admissible sets} to prove the existence of tight wavelet frame sets in  $\Bbb F_q^d$, when $d=2$,  $q$ prime and  $q\equiv 3$ (mod 4).

 In the sequel we shall assume that the subset $E$ is non trivial, i.e., $1<\sharp E<q^d$. 
\section{Sufficient Conditions for Tight Wavelet Frame Sets}\label{Sec:Parseval Wavelet Frames}
 Here we first review the Fourier transform
  on $\Bbb F_q^d$.  If $f$ is a function on  ${\mathbb F}_q^d$, then the
 Fourier coefficients, $\hat f(\xi)$,  of  $f$  are given by
$$
\widehat{f}(\xi) = q^{-d}\sum_{m\in{\mathbb F}_q^d} a_m\overline{\chi_m(\xi)} \quad \xi\in {\mathbb F}_q^d
$$
where
 $\chi_m(\xi)=
e^{2\pi i \frac{m \cdot\xi}{q}}$ is the character and  $m\cdot \xi$ is the usual inner product and $a_m= f(m)$.  The function $\hat f$ is called the Fourier transform of $f$.
With the above notations,
$$f(x) = \sum_{m\in \Bbb F_q^d} c_m \chi_m(x), $$
where $c_m = \hat f(m)$. By the above definition, the Fourier transform $\mathcal F: L^2({\mathbb F}_q^d) \to L^2({\mathbb F}_q^d)$ given by $f\to \hat f$ is a unitary map. For any $g\in L^2({\mathbb F}_q^d)$, we shall denote by $\check g$ the inverse Fourier transform of $g$.

Let Aut$({\mathbb F}_q^d)$ be the set of all automorphisms on ${\mathbb F}_q^d$.  Let $\psi$ be a function defined on  $\mathbb F_q^d$. For $a\in {\text Aut}({\mathbb F}_q^d)$ and $t\in \Bbb F_q^d$ define  the associated dilation  and translation operators  $\delta_a$ and $\tau_t$ by 
$$\delta_a \psi(x)= \psi(a x) ,$$
and   
$$\tau_t \psi(x) = \psi(x-t), $$ 
respectively. 
These operators are {\it unitary}.



\begin{lemma} Let $\psi: \Bbb F_q^d \to \Bbb C$. Then
for given automorphism $a\in \mbox{Aut}({\mathbb F}_q^d)$ and  $t\in \Bbb F_q^d$   we have
\begin{align}
\widehat{\delta_a\tau_t \psi}(m)= \overline{\chi_{a^{-1} t}(m)}\hat \psi(a^\ast m),
\end{align}
where $a^\ast= (a^t)^{-1}= (a^{-1})^t$ is the inverse transpose of $a$. If $\hat\psi= 1_E$, then
\begin{align}
\widehat{\delta_a\tau_t \psi}(m)= \overline{\chi_{a^{-1} t}(m)}1_{a^t(E)}(m) .
\end{align}
\end{lemma}
 
\begin{proof} Let $t\in\Bbb F_q^d$ and  $a\in \mbox{Aut}(\Bbb F_q^d)$.   By applying the Fourier  transform and  using a change of variable in the definition of the Fourier transform, for all    $ m \in \Bbb F_p^d$  we have
 
\begin{align}\label{Fourier transform of translations}
 \widehat{\tau_t\psi}(m) &=q^{-d}\sum_{n\in \Bbb F_q^d} \psi(n-t) \overline{\chi_m(n)} \\\notag
 & =q^{-d}\sum_{n\in \Bbb F_q^d} \psi(n) \overline{\chi_m(
 n+t)} \\\notag
 &=
\overline{\chi_m(t)}
 \left(q^{-d}\sum_{n\in \Bbb F_q^d} \psi(n) \overline{\chi_m(n)}\right)  \\\notag
&= \overline{\chi_m(t)}
   \hat \psi(m)\\\notag
   &=  \overline{\chi_t(m)}
   \hat \psi(m),\\\notag
 \end{align}
 and
 \begin{align}\label{Fourier transform of dilation}
\widehat{\delta_a \psi}(m) &= q^{-d} \sum_{n\in \Bbb Z_p^d} \psi(a n) \overline{\chi_m(n)} \\\notag
  &=q^{-d}
\sum_{n\in \Bbb Z_p^d} \psi(n) \overline{\chi_m(a^{-1}n)}\\
& = q^{-d}\sum_{n\in \Bbb Z_p^d} \psi(n) \overline{\chi_{a^*m}(n)}\\
  &= \hat\psi(a^\ast m) .
\end{align}
Now, a combination of  (\ref{Fourier transform of dilation}) and  (\ref{Fourier transform of translations})   yields the assertion of the lemma:
 \begin{align}\label{Fourier transform of dilation and translation}
 \widehat{\delta_a\tau_t \psi}(m)  =\widehat{\tau_t \psi}(a^\ast m) =  \overline{\chi_{a^\ast m}(t)}\hat \psi(a^\ast m)= \overline{\chi_{a^{-1} t}(m)}\hat \psi(a^\ast m) .
  \end{align}
  By the equality $1_E(a^\ast m) = 1_{a^t(E)}(m)$, the second part of the lemma also holds true.
 \end{proof}
 %
  %

The proof of the  following result  is straightforward  using the  Fourier transform.
\begin{lemma}\label{ONB-FT-ONB}
Let $\mathcal A\subseteq Aut(\Bbb F_q)$ and $T\subseteq\Bbb F_q^d$, and the family
$$\{\delta_a \tau_t \psi: \ a\in \mathcal A, \ t\in T\}$$
is an orthonormal basis for $L^2(\Bbb F_q^d)$ if and only if the family
$$\{\overline{\chi_{a^{-1}t}(m)}\hat \psi(a^\ast m): \ t\in T, a\in \mathcal A\}$$
is an orthonormal basis  for $L^2(\Bbb F_q^d)$. Here, $m$ is the generic variable.
\end{lemma}

\medskip

The next result  proves the existence of no Parseval  wavelet frame for $L^2(\Bbb F_q^d)$ generated by  $\psi:= \check{1_E}$. 
\begin{theorem}\label{NO-ON-wavelet-basis} There exists no  non-empty  subset  $E\subsetneq\Bbb F_q^d$
 such that $\psi:= \check{1_E}$, the inverse Fourier transform of the indicator function $1_E$,  is the generator of  a Parseval wavelet frame  for $L^2(\Bbb F_q^d)$.
\end{theorem}
\begin{proof} We shall prove this theorem by a contradiction argument. Assume that  for a subset $E$    there is an  automorphism set $\mathcal A\subset  \mbox{Aut}(\Bbb F_q^d)$ and a subset $\Lambda$ of $\Bbb F_q^d$ such that the set
$$\{ \delta_a\tau_\lambda \check{1_E}: \ a\in \mathcal A, \ \lambda\in \Lambda\}$$
is  a Parseval frame for $L^2(\Bbb F_q^d)$. By Lemma \ref{ONB-FT-ONB}, this is equivalent to say that  the family
 \begin{align}\label{xy}
 \{\overline{\chi_{a^\ast m}(\lambda)} 1_{a^t(E)}(m): \ \lambda\in \Lambda, a\in \mathcal A\} = \{\overline{\chi_{a^{-1} \lambda}(m)} 1_{a^t(E)}(m): \ \lambda\in \Lambda, a\in \mathcal A\}
 \end{align}
is a Parseval frame for $L^2(\widehat{\Bbb F_q^d})=L^2({\Bbb F_q^d})$. (Here,   $a^t$ is the transpose of $a$.)
  Let $\hat g= 1_{\{\vec 0\}}\in L^2(\widehat{\Bbb F_q^d})$ be the indicator function for the set $\{\vec 0\}$. Then
 \begin{align}\label{calculation}
 1= \|\hat g\|^2&= \sum_{a\in\mathcal A, \lambda\in\Lambda} |\langle g,  \delta_a\tau_\lambda \psi\rangle|^2\\\notag
 &= \sum_{a\in\mathcal A, \lambda\in\Lambda} |\langle \hat g, \widehat{\delta_a\tau_\lambda \psi}\rangle|^2
 \\\notag
 &= \sum_{a\in\mathcal A, \lambda\in\Lambda} |\sum_{m\in \Bbb F_q^d} \hat g(m) 1_{a^t(E)}(m) \overline{\chi_{a^\ast m}(\lambda)}|^2
  \\\notag
 &= \sum_{a\in\mathcal A, \lambda\in\Lambda} |1_{a^t(E)}(0)|^2\\\notag
&= \sharp(\Lambda)   \sum_{a\in\mathcal A} |1_{a^t(E)}(0)|^2 .
 \end{align}

 Here, we shall consider two cases: If $0\in E$, then $0\in a^t(E)$ for all $a\in \mathcal A$. Thus the above calculation implies that
  $ 1= \sharp(\mathcal A) \sharp(\Lambda). $
This means that the wavelet system must have only one element. Let $\mathcal A=\{a\}$ and $\Lambda=\{\lambda\}$. Then all the vectors in $L^2(\Bbb F_q^d)$ must be a constant multiple of $\delta_a \tau_\lambda \psi$. We show that this is not the case since    $E\neq \Bbb F_q^d$.   Let  $f\neq 0$ in  $L^2(\Bbb F_q^d)$ such that $\supp({\hat f}) \cap  a^t(E)=\emptyset$. Such function exists since $E$ is not the whole $\Bbb F_q^d$.  Then  there is no cosntant $c$ such that   $f= c \delta_a \tau_\lambda \psi$. This shows that the Parseval wavelet frame can not have only one element when $E\neq \Bbb F_q^d$, thus $ \sharp(\mathcal A) \sharp(\Lambda)>1$.

Now let us now assume that $0\not\in E$.  By the equations in (\ref{calculation}) we obtain  $1=0$ which is impossible.  This completes the proof our assertion.
\end{proof}

  As a result of
 Theorem \ref{NO-ON-wavelet-basis},  we have the following corollary.
 
 \medskip
 
  \begin{corollary}\label{cor:no ONB wavelet set}
There is no orthonormal basis of form $\{\delta_a \tau_\lambda\psi\}_{a\in \mathcal A, \lambda\in \Lambda}$ for $L^2(\Bbb F_q^d)$ where $\hat \psi=1_E$ and $E\subsetneq \Bbb F_q^d$.
 \end{corollary}

{\it Remark.} Notice when $E=\Bbb F_q^d$,   by a similar calculation  (\ref{calculation}) for $\hat g= 1_{\Bbb F_q^d}$,  we can conclude that  the translation set  $\Lambda$ must contains $0$ and $\sharp(\mathcal A)=1$. Then by the divisibility we must have   $\Lambda=\Bbb F_q^d$. This implies that the set   $\{\tau_\lambda \psi: \lambda\in \Bbb F_q^d\}$,  with $\hat \psi= 1_{\Bbb F_q^d}$,   forms an orthonormal basis for $L^2(\Bbb F_q^d)$. But   this    is already well-known by the Fourier transform. Therefore, it is reasonable  to assume  in Theorem \ref{NO-ON-wavelet-basis}   that $E\neq \Bbb F_q^d$.  
\medskip

As we observed above,   Corollary \ref{cor:no ONB wavelet set}
  implies  the existence of no orthonormal basis of form  (\ref{xy}) for $L^2(\Bbb F_q^d)$. However, in the following theorem, we prove that if we choose $E$
   appropriately  in $\Bbb F_q^d$, then the family
 (\ref{xy}) for $E^*= E\setminus \{0\}$ is a tight wavelet frame  for $L^2(\Bbb (F_q^d)^\ast)$.\\
  For the rest, we use the notation  $Y:= (F_q^d)^\ast$.

\medskip

 \begin{definition} Given $E$ and $L$ subsets of $\Bbb F_q^d$, we say $(E,L)$ is a (tight) frame spectral pair if the set of ``exponentials" $\{ \chi_{l}:  l\in L\}$ is  a  (tight) frame for $L^2(E)$.
 \end{definition}

\medskip

 \begin{theorem}\label{P.F}
 Let $E$ and $L$ be subsets of $\Bbb F_q^d$  and $(E,L)$ is a spectral pair. Assume that  $E^*$ is a multiplicative tiling set with respect to a set of  automorphisms $\mathcal A\subset  A ut(\Bbb F_q^d)$.
 Then the following hold true:

 \begin{itemize}
 \item[(1)]  $(E^*,L)$ is tight frame spectral pair with the frame bound  $\sharp(E)$.
 \item[(2)] $\forall \ a\in \mathcal A$, \ $(a(E^*),(a^{-1})^t(L))$ is tight frame spectral pair with the frame bound  $\sharp(E)$.
 \item[(3)] The family   $\{ 1_{a(E^*)}\chi_{(a^{-1})^t(l)}:  l\in L, a\in \mathcal A\}$ is a tight frame for $L^2(Y)$ with the frame bound $\sharp(E)$, where $Y=(\Bbb F_q^d)^*$.
 \item[(4)] If $0\not\in E$, then    $\{ (\sharp E)^{-1/2} 1_{a(E)}\chi_{(a^{-1})^t(l)}:  l\in L, a\in \mathcal A\}$ is an orthonormal basis for $L^2(Y)$. 
 \end{itemize}
 \end{theorem}

 \begin{proof} Assume that $E$ has a spectrum $L$. Then $\{ \sharp(E)^{-1/2}\chi_{l}:  l\in L\}$ is an orthonormal basis for  $L^2(E)$.
 To proof  (1),  note that the map $f\to f 1_{E^*}$ is a projection of $L^2(E)$ onto $L^2(E^*)$, thus the image of the orthonormal basis 
 $\{\sharp(E)^{-1/2} \chi_{l}:  l\in L\}$ by this map  is a Parseval frame for $L^2(E^*)$. 
  This proves the statement (1). To prove (2), we instead prove the following: Let $F\subset \Bbb F_q^d$ and $\{ \chi_{l}:  l\in L\}$ be a frame for $L^2(F)$ with the frame bounds $0<A\leq B <\infty$. Then $\{ \chi_{(a^{-1})^t(l)}:  l\in L\}$ is  a frame for $L^2(a(F))$ with the unified frame bounds,  $A$ and $B$. To prove this,   define the map $T_a: L^2(F)\to L^2(a(F))$ by $f\to f\circ a^{-1}$. The image of 
   $\{\chi_{l}:  l\in L\}$ under  $T_a$  is $\{ 1_{a(F)}\chi_{(a^{-1})^t(l)}:  l\in L\}$ and the map  
   is unitary. Therefore   $\{ 1_{a(F)}\chi_{(a^{-1})^t(l)}:  l\in L\}$ forms a frame for $L^2(a(F))$ with the same frame  bounds. 
  

  %
  %

     To prove (3), note that by the assumption on the multiplicative tiling property of $E^*$ we have 
     $$L^2(Y)= \oplus_{a\in \mathcal A} L^2(a(E^*)). $$
     To complete (3),  we shall  prove the following, instead.

     Let $X$ be measurable set and $A$ be an index set such that  $X= \cup_{a\in A} X_a$  (disjoint). Assume that for all $a\in A$,  $L^2(X_a)$ has a tight frame $\{f_{n,a}: \ n\in I_a\}$ with frame bound $C$.  We claim that the family $\{f_{n,a}: \  a\in A, \ n\in I_a\}$ is a tight frame for
     $L^2(X)$ with the frame bound $C$. To prove that, let $g\in L^2(X)$. Sine $\{X_a: a\in A\}$ is a partition for $X$, then   $g= \oplus_{a\in A} g_a$ ,   $g_a:= g1_{X_a}$,  and we have 
     \begin{align}
     \|g\|^2 &=  \sum_{a\in A} \|g_a\|_{L^2(X_a)}^2\\\notag
     &=  \sum_{a\in A} \left(C^{-1}\sum_{n\in I} |\langle g_a, f_{n,a}\rangle_{L^2(X_a)}|^2\right)\\\notag
       &=C^{-1}  \sum_{a\in A, n\in I_a}    |\langle g, 1_{X_a}f_{n,a}\rangle_{L^2(X)}|^2 .
     \end{align}
     This completes the proof of the assertion,  thus (3).  For (4), note that $E=E^*$ when $0\not\in E$. Then the  proof  can be obtained directly from the fact that any Parseval frame with normalized frame elements is an orthonormal basis.
  \end{proof}
\medskip

  We state the following result from \cite{IMP15} which characterizes tiling set and spectral sets in $\Bbb F_q^2$ and proves the  Fuglede conjecture for $\Bbb F_q^2$.

\medskip

  \begin{theorem}[Fuglede conjecture for  $\Bbb F_q^2$]\label{Fuglede Conj. holds}
  A subset $\emptyset\neq E$ of $\Bbb F_q^2$ tiles   $\Bbb F_q^2$ with its translations if and only if there is a set $L\subseteq \Bbb F_q^2$ such that $(E,L)$ is  a spectral pair.
  \end{theorem}

\medskip
 As the corollary of 
  Theorems \ref{Fuglede Conj. holds}  and  Theorem \ref{P.F} (3) we have the following result. 
 
  \begin{corollary} Assume that $E$  is a translation tiling for  $\Bbb F_q^2$ and  $E^*$ is a multiplicative tiling  with respect to   the  automorphisms $\mathcal A\subset A ut(\Bbb F_q^2)$. Then there is a  set $L$ in $\Bbb F_q^2$ such that the  family   $$\{(\sharp E)^{-1/2}\chi_{a^{-1}(l)} 1_{a(E^*)}:  l\in L, a\in \mathcal A\}$$
is a Parseval frame for $L^2(Y)$. The system is an orthonormal basis if $0\not\in E$. 
  \end{corollary}

Given a subset $F$ of $\Bbb F_q^d$ we say a function $f \in L^2(\Bbb F_q^d)$ belongs to  $PW_F$ (The Paley-Wiener space) if its Fourier transform $\hat f$ has support in $F$. Then
 \begin{align}\label{def-PW}
 PW_F:= \{ f\in L^2(\Bbb F_q^d): \ \hat f(m)=0 \ \forall m\not\in F\}.
 \end{align}
    Note that if $F= (\Bbb F_q^d)^*$, then $PW_Y$ contains the  all functions
 $f\in L^2(\Bbb F_q^d)$ for which  $\sum_{m\in \Bbb F_q^d} f(m) = 0$.

 Our next result     is  analogous with Theorem 1.1. in \cite{Wang} in $\Bbb F_q^d$.   For  $\mathcal A\subset \mbox{Aut}(\Bbb F_q^d)$,  we denote by $\mathcal A^t$ the set of transpose of all matrices in $\mathcal A$. \\
%
%
  %

  \begin{theorem}\label{PWF} Assume that  $E\subseteq \Bbb F_q^d$  has a  spectrum $L$ and  $\emptyset\neq E^*$ is multiplicative tiling with respect to $\mathcal A^t$ for some $\mathcal A\subseteq \Bbb A ut(\Bbb F_q^d)$.  Take  $ \psi := (\sharp E)^{-1/2} \check{1_{E^*}}$  and  $Y:=  (\Bbb F_q^d)^*$.  Then the family $\mathcal W:=\{\delta_a \tau_l \psi: \ a\in \mathcal A, l\in L\}$ is an Parseval frame   for $PW_Y$. The system $\mathcal W$ is an orthonormal basis if $E=E^*$.  Conversely, if $\psi=(\sharp E)^{-1/2}\check{1_E}$ is a Parseval wavelet for $PW_Y$ with respect to  dilation set $\mathcal A$ and translation set $L$, then $Y= \cup_{a\in \mathcal A} a^t(E^*)$. The sets $a^t(E^*)$ are disjoint, then  $(E^*, L)$ is a tight frame spectral pair. 
\end{theorem}

 \begin{proof}   Assume that $(E, L)$ is a spectral pair and $E^*$ is a multipilicative tiling with respect to $\mathcal A^t$ for some $\mathcal A\subseteq \mathcal A ut(\Bbb F_q^d)$.
By an application of the Fourier transform,  to prove that   $ \mathcal W$  is a  Parseval frame for $PW_Y$ is equivalent to say that $\widehat{\mathcal W}:=\{(\sharp E)^{-1/2} \chi_{a^{-1}(l)1_{a^t(E^*)} }:  l\in L, a\in \mathcal A\}$ is a  Parseval frame for $L^2(Y)$.
  By Theorem \ref{P.F} (3)  we know that $\widehat{\mathcal W}$  is a Parseval frame for $L^2(Y)$.    Therefore by    inverse  of  the Fourier transform, which is  unitary, we   conclude that the wavelet system $\mathcal W$  is a Parseval frame for $PW_Y$ and this completes the proof of "$\Leftarrow$".

Now assume that  $\mathcal W$ is a Parseval frame for $L^2(Y)$.   To prove the union of    the sets $a^t(E^*),  \ a\in \mathcal A$, covers $F$, we use a contradiction argument. Assume that   there is a non empty set $M\subsetneq Y$ such that $M\cap a^t(E^*) = \emptyset$ for all $a\in \mathcal A$. Take $f:= 1_M$. Then
  \begin{align}
  \|f\|^2 =(\sharp E)^{-1/2}  \sum_{a\in\mathcal A, l\in L} |\langle f, 1_{a^t(E^*)}\chi_{a^{-1}(l)}\rangle|^2
  \end{align}
  Since $M\cap a^t(E^*)= \emptyset$ for all $a\in \mathcal A$, then the right side in the preceding equation must be  zero, while the left side is $\sharp(M)$. This is a contradiction to our assumption that $M$ is non-empty, therefore $Y= \cup_{a\in \mathcal A} a^t(E^*)$. 
  To show that the pair $(E^*, L)$ is a tight frame spectral pair, note that by the  hypothesis on the disjointness of the sets $ a^t(E^*)$, for any $a\in \mathcal A$ the system $\{\chi_{a^{-1}(l)}:  l\in L\}$ is a tight frame for $L^2(a^t(E^*))$ with the frame bound $(\sharp E)^{-1/2}$. The dilation operator $T_a: L^2(a^t(E^*)) \to L^2(E^*)$ given by  $f\to f\circ a^t$,  $f\circ a^t(x) = f(a^t(x))$,  $x\in E^*$,   is  unitary, therefore   $\{\chi_{l}:  l\in L\}$,    image of $\{\chi_{a^{-1}(l)}:  l\in L\}$ under $T_a$,  is a tight frame for $L^2(E^*)$ with the unified  frame bound and we are done. 
  \end{proof} 
  
 The disjointness of the sets $a^t(E^*), \ a\in \mathcal A$,  in Theorem  \ref{PWF} can be obtained under some additional assumptions on $\mathcal W$.

  \begin{theorem}\label{orthogonal system}
   If the system $\mathcal W$ is an orthogonal basis for $PW_Y$ and $0\in L$, then the sets $a^t(E^*), \ a\in \mathcal A$, are mutual disjoint and $(E^*, L)$ is a  spectral pair: 
   \end{theorem} 
  
  \begin{proof} 
  Now assume that  $\mathcal W$ is an orthogonal basis and $0\in L$.   Take $l=0$. Then the functions in the  family  
  $\{1_{a^t(E^*)}: a\in \mathcal A\}$ are orthogonal and 
 $\sharp (a_1^t(E^*) \cap  a_2(E^*))=0$ for any distinct $a_1$ and $a_2$. From the other side, the exponentials  $\{\chi_{a^{-1}(l)}:  l\in L\}$  are orthogonal basis for $L^2(a^t(E^*))$.  By a similar argument as above,  we conclude that  the family  $\{\chi_l:  l\in L\}$  is also an orthogonal basis for  $L^2(E^*)$. This completes the proof of the theorem.   
 \end{proof}
 
 {\color{red} Question:  Must $0\in E^c$ when the system is an orthogonal system? } 
 
%
We conclude this section with an example of a tight wavelet frame  associated to a set $E$ and automorphisms $\mathcal A$,    where the sets  $a^t(E)$, $a\in \mathcal A$, are not necessarily disjoint. Let $\mathcal W$ be a tight wavelet   frame for $PW_Y$ with frame bound $A$,  where $\hat\psi= 1_E$. Take $\mathcal W_1= \mathcal W\cup \mathcal W$. Then $\mathcal W_1$ is a tight frame  for $PW_Y$ with the frame  bound $A/2$ and the sets $a^t(E^*)$ are not disjoint. 

 \section{Existence of Multiplicative  tiling sets  in  ${\mathbb F}_q^d$}\label{Existence of wavelet sets}
  Let $q$ be an odd prime and $d\geq 1$. In this section we shall prove the existence of non-trivial (non-trivial here simply  means that the set $E$ is neither the whole space nor one point) multiplicative tiling sets in the finite vector space $\Bbb F_q^d$ when $q\equiv 3$ (mod 4). 

  \medskip

When $d=1$, non-trivial multiplicative tiling set on ${\mathbb F}_q$, $q$ odd, exists. Notice that ${\mathbb F}_q$ can be identified as $\{\frac{-(q-1)}{2},...,-1,0,1,...,\frac{q-1}{2}\}$ and $q-1$ is an even number. Define the automorphisms $\alpha_1$ and $\alpha_2$ to be $\alpha_1(x) = x$ and $\alpha_2(x) = -x$. Take  $E = \{1,...,(q-1)/2\}$. Then we immediately  see that
$$
{\mathbb F}_q\setminus\{0\} = \alpha_1(E)\dot\cup\alpha_2(E).
$$
The union is disjoint and this yields naturally a non-trivial multiplicative tiling. If $d>1$,  it is not immediately clear that why non-trivial multiplicative tiling sets  exist. We first prove this for  $d=2$.

\medskip

Let us first consider the problem of multiplicative tiles on ${\mathbb R}^2$ and gain some motivations. Indeed, if we define $R_{\theta}$ be the rotation matrix of angle $\theta$ and $E_p$ be the sector without the origin with aperture $2\pi/p$, then  ${\mathbb R}^2\setminus\{0\}$ is naturally partitioned into
$$
{\mathbb R}^2\setminus\{0\} = \bigcup_{k=0}^{p-1}R_{2\pi k/p}(E_p).
$$
We can make the set compact by considering annulus of sectors with inner and outer radii equal to  $1$ and $2$, respectively, and taking also dilation matrices into account. However, sectors can never be a translational tile and hence we cannot produce wavelet sets on ${\mathbb R}^2$ using sectors. Nonetheless, we will see our construction of wavelet sets are produced by rotation on the finite field and it can also be a translational tile on $\Bbb F_q^2$. 

\medskip

For $r\in {\mathbb F}_q$,  we consider the circle $S_r$ on ${\mathbb F}_q^2$ with radius $r$ as follows:

  $$
S_r: = \left\{\left(
           \begin{array}{c}
             x \\
             y \\
           \end{array}
         \right)\in{\mathbb F}_q^2: \ x^2+y^2 = r.
\right\}
$$

We recall that $a$ is a {\it quadratic residue} (mod $q$) if $x^2\equiv a$ (mod $q$) has a solution in ${\mathbb F}_p$, otherwise, it is called a {\it quadratic non-residue}. On ${\mathbb F}_q$, there exist exactly $(q-1)/2$ of non-zero quadratic residue  and $(q-1)/2$ are quadratic non-residue. By studying the tiling properties of the quadratic residues, we have the following lemma:

\begin{lemma} Let $q\equiv3$  (mod 4). Then

$$
\#S_r= \left\{
         \begin{array}{ll}
         1, & \hbox{if $r=0$;} \\
           q+1, & \hbox{if  $r\neq 0$} \\
           \end{array}
       \right.
 $$
\end{lemma}

\begin{proof}
By Theorem 1.2 in \cite{ME}, every quadratic residue [non-residue] can be written as a sum of two quadratic residues [non-residues] in exactly $d_q-1$ ways, and every quadratic residue [non-residue] can be written as a sum of two quadratic non-residues [residues] in exactly $d_q$ ways, where
$$
d_q =  \frac{q+1}{4} , \quad  \text{when}  \quad q\equiv3  \  (mod \ 4)
$$

 Suppose that $r$ is a quadratic residue mod $q$. Then each sum of $r$ as quadratic residues $a_1+a_2$ induces four points in $S_r$. Indeed, there are $a_1= x^2$ has two solutions $x_1, x_2$ and $a_2 = x^2$ has two solutions $y_1, y_2$. Thus there are four distinct points $(x_1,y_1)$, $(x_1,y_2)$, $(x_2,y_1)$ and $(x_2,y_2)$.
 Furthermore, as $r = x^2$ also has two solutions $z_1,z_2$. It induces 4 more solutions $(z_1,0)$, $(z_2,0)$, $(0,z_1)$ and $(0,z_2)$ on the axes.
 Hence, by the theorem,  when $q\equiv3$ (mod 4) we have
\begin{equation}\label{eq1}
\#S_r = 4(d_q-1) +4 = 4d_q =
          q+1.
\end{equation}
 Suppose that $r$ is a quadratic non-residue of $q$. Then  $r$ can be written  as exactly $d_q$ ways as sum of quadratic residues. As each sum induces 4 pairs, $\#S_r = 4d_q$ which is the same answer as in  (\ref{eq1}).

 Finally, it follows directly that
$$
\#S_0 = q^2-\sum_{r=1}^{q-1}\#S_r =q^2-(q-1)(4d_q) =
  1 ,
$$
as required.
 \end{proof}

 In the following lemma    we prove that there exists  an orthogonal matrix such that the multiplication of its exponents with a vector in a circle $S_r$, $r\neq 0$, generates the circle.

 \begin{lemma}\label{lem2} Suppose  that $q\equiv 3$ (mod 4).
 Then there exists an orthogonal matrix
$$
R = \left(
  \begin{array}{cc}
    a & -b \\
    b & a \\
  \end{array}
\right)
$$
such that $a^2+b^2=1$ (mod $q$), $R^{q+1}=I$,  and for any ${\bf e}\in S_r$, $R{\bf e},R^2{\bf e},...,R^{\#S_r}{\bf e}$ generates $S_r$ for all  $r\ne 0$.
\end{lemma}

\begin{proof} Since $q\equiv 3$ (mod 4), then $-1$ is not a square in ${\mathbb F}_q$.  In particular, we define $i$ to be the imaginary solution of $i^2=-1$ (mod q). Thus, we can identify $\left(
                                                           \begin{array}{c}
                                                             x \\
                                                             y \\
                                                           \end{array}
                                                         \right)$ in ${\mathbb F}_q$ as $x+yi$.
In this sequel, ${\mathbb F}_q^2$ is isomorphic to the finite field of $q^2$ elements, denoted as ${\mathbb F}_{q^2}$. Note that the multiplicative group ${\mathbb F}_{q^2}^{\times}$ is a cyclic group and the circle
$$
S_1  = \{x+yi: x^2+y^2=1\}
$$
is a subgroup of ${\mathbb F}_{q^2}^{\times}$. As $\#S_1 = q+1$,  there exists $a+bi\in S_1$ such that   $S_1 = \{1, a+bi, (a+bi)^2,...,(a+bi)^q\}$. In other words, $a+bi$ generates the group $S_1$.

Then it follows that for any ${\bf e}\in S_r$ we can write ${\bf e} = c+di$,  and we have
$$
S_r = \left\{c+di, (a+bi)(c+di),...,(a+bi)^q (c+di)\right\}.
$$
 Define the matrix $R$ by $ \left(
  \begin{array}{cc}
    a & -b \\
    b & a \\
  \end{array}
\right).$ Observe that  $R{\bf e} = (ac-bd, \ bc+ad)^T$ in ${\mathbb F}_{q}^2$ is equal to $(ac-bd)+(ad+bc)i = (a+bi)(c+di)$ in $\Bbb F_{q^2}$. Then we have
$$
S_r = \left\{{\bf e},R{\bf e},...,R^q{\bf e}\right\} .
$$
This completes the proof of the lemma.

\end{proof}

Our next result   proves the existence of  non-trivial  multiplicative tiling sets  ${\mathbb F}_q^d$ for $d> 1$.

\medskip

\begin{theorem}\label{theorem_multiplicative}
There exists multiplicative tiling set  $E$ in ${\mathbb F}_q^d$ for $q\equiv 3$ (mod 4).
\end{theorem}
\begin{proof}
We first consider $d=2$. Take the automorphisms to be $\{I,R,R^2,...,R^q\}$ where $R$ is defined in Lemma \ref{lem2}. Define the set $E$ by taking one point from each $S_r$, for $r\ne 0$. Then Lemma \ref{lem2} shows that ${\mathbb F}_q^2\setminus \{ 0\} = E\cup R(E)\cup...\cup R^{q}(E)$ and the unions are disjoint since each $R^j(E)$ intersects $S_r$ exactly once.


If $d>2$, we may take $\widetilde E:=E\times {\mathbb F}_q^{d-2}$ and  $\widetilde R := \left(
                                                                                    \begin{array}{cc}
                                                                                      R & O \\
                                                                                      O & I \\
                                                                                    \end{array}
                                                                                  \right), 
$
where $I$ is $(d-1)\times(d-2)$ identity matrix. 
Then $\widetilde E$ is a multiplicative tiling set  associated to the automorphisms $\{\widetilde R^j: \  1\leq j\leq q+1\}$.
\end{proof}

We will call the multiplicative tiling in Proposition \ref{theorem_multiplicative} {\it rotational tiling}.
Therefore, any rotational tiling is a set of $q-1$ elements which has only one vector in the  intersection with any circle of non-zero radius. In the sequel, we will consider the wavelet sets  originated by the rotational tilings.

\section{Construction of tight wavelet frame sets}\label{existence of admissible sets}

  As we will see in this section, the  construction of  tight wavelet frame   sets on ${\mathbb F}_q^2$ requires us to find a set $E$ such that $0\in E$, $E^*$  is a multiplicative  tiling and
 $E$ tiles $\mathbb F_q^2$ by translations.    Here we consider rotational tilings.
 We recall the following characterization of  translational tiles on ${\mathbb F}_p^2$.

\medskip

\begin{theorem}\cite{IMP15}\label{ThIMP15}
Let $E$ be a set that tiles ${\mathbb F}^2_q$ by translation. Then $\sharp E = 1, q$  or $q^2$ and $E$ is a graph if
$\sharp E = q$, i.e.
$$
E= \{x{\bf e}_1+f(x){\bf e}_2: \  x\in {\mathbb F}_q\}
$$
for some basis ${\bf e}_1,{\bf e}_2$ in ${\mathbb F}_q^2$ and function $f:{\mathbb F}_q\rightarrow{\mathbb F}_q$ .
\end{theorem}

Note that,  by our construction  in Proposition \ref{theorem_multiplicative},  a  rotational tiling set   has  $q-1$ elements. Therefore,  by the classification of translational tiles in Theorem \ref{ThIMP15}, to construct tight wavelet frame sets,  we must consider  graph of functions defined on $\mathbb F_q$ that are simultaneously rotational tiling sets. We will provide a systematic way to construct such sets as  we prove   Theorem \ref{Thm:admissible sets}.   First we need some key lemmas.

 \medskip

\begin{lemma}\label{k} There exists $k\in  \mathbb F_q$, $0< k\leq  \frac{q-1}{2}$ such that $1+k^2$ is a quadratic non-residue.
\end{lemma}
\begin{proof} Assume that such $k$ dose not exists. This means that $1+k^2$ are  quadratic residues for all $0< k\leq  \frac{q-1}{2}$.  We note that all $1+k^2$ are in distinct residue classes (mod $q$), otherwise $1+k^2\equiv 1+k'^2$ would imply $k = k'$ or $k= -k'\equiv q-k'$ (mod $q$). The latter is not possible since $0<k,k'\le \frac{q-1}{2}$.
Since $\sharp QR=\frac{q-1}{2}$ and $1+k^2$ are distinct for different  $k$, then we must have  $QR= \{1+k^2: 0< k\leq  \frac{q-1}{2}\}$.  However, we also know that  $1\in QR$. Then   for some $k$  we have  $1+k^2\equiv 1$ (mod q). This implies that $k=0$ or $k=q$, which contradicts the assumption.
\end{proof}

As an example for $k$, if $q=7, 19$, or  $23$, then $k$ equals to   $1, 1$, or $2$, respectively.
The existence of $k$
  in Lemma \ref{k} allows us to represent  the quadratic non-residue numbers as follows.
\medskip
\begin{lemma}\label{QNR}  For the $k$ defined in Lemma \ref{k},
$$
\mbox{QNR} = \{(1+k^2)x^2: (q+1)/2\le x\le q-1\},
$$
and $QR = \{x^2: 0\le x\le (q-1)/2\}$ and
\end{lemma}
\medskip
\begin{proof}
There is nothing to prove about the
  statement about QR. For the  statement about $QNR$, due to    the multiplicative property of Legendre symbol we have 
$$
\left(\frac{(1+k^2)x^2}{q}\right) = \left(\frac{1+k^2}{q}\right)\left(\frac{x^2}{q}\right) = (-1)(1) = -1
$$
(Recall that Legendre symbol is equivalent to the  Euler\rq{}s criterion  and $(a/q)\equiv a^{\frac{q-1}{2}}=1$ if $a$ is a quadratic residue and $(a/q)=-1$ if $a$ is a quadratic non-residue). Hence, all $(1+k^2)x^2$ are quadratic non-residue. Moreover, they are all distinct since if $(1+k^2)x^2 = (1+k^2)y^2$ and $x\ne y$, then $x=q-y$, which means $x,y$ can't be in $\{(q+1)/2,...,q-1\}$ at the same time. Hence,  the set $\{(1+k^2)x^2: (q+1)/2\le x\le p-1\}$ contains all $(q-1)/2$ quadratic non-residues. This proves the second statement.
\end{proof}
\medskip

The main result of this section follows.

 \begin{theorem}\label{Thm:admissible sets}
 Assume that $q$ is an odd prime congruent to $3$ (mod 4). Then there exists a subset $E$ of $\Bbb F_q^d$ such that  $E$  is  a translational tiling,  has a spectrum and  $E^*$ is a multiplicative tiling in    ${\mathbb F}_q^d$.
 \end{theorem}
\begin{proof}
We first prove the case $d=2$. By Theorem \ref{ThIMP15} and Theorem \ref{theorem_multiplicative},  it is sufficient to construct a set $E$ in ${\mathbb F}_q^2$ for $q\equiv 3$ (mod 4) which is a graph of the form
$$
E= \{(x,f(x)):x\in{\mathbb F}_q\}
$$
for some function  $f:{\mathbb F}_q\rightarrow{\mathbb F}_q$ and $\sharp (E\cap S_r)= 1$ for all $r\in{\mathbb F}_q$.
 For this, we construct $E$ with the following properties:
\begin{itemize}
\item $\vec0= (0,0) \in E$
\item  $(x,0)\in E$ where $0<x\leq \frac{q-1}{2}$
\item  For $k$ in Lemma \ref{k}, let   $(x,kx)\in E$ where $\cfrac{q+1}{2} \leq x< q$. 
\end{itemize}
 Set $E$  is clearly a graph of function  $f:{\mathbb F}_q\rightarrow{\mathbb F}_q$  with  $f(0)=0$, $f(x)= 0$ when    $0<x\leq \frac{q-1}{2}$, and $f(x) = kx$  when   $\frac{q+1}{2} \leq x< q$.
 So, it tiles $\mathbb F_q^2$ by translations with respect to some coordinate system and the tiling partner
 $A= \{(0,t): \ t\in \mathbb F_q\}$, and  by Theorem \ref{ThIMP15} it has a spectrum. 
  We show that $E$  is a rotational tiling by showing that $\sharp(E\cap S_r)=1$, $r\neq 0$. Indeed, if $r$ is a quadratic residue, then there exists unique $x$ satisfying $0<x\leq \frac{q-1}{2}$ such that $x^2+0^2 = r$. If $r$ is a quadratic non-residue, then Lemma \ref{QNR} implies the existence of the unique $x$ satisfying $\frac{q+1}{2}<x\leq q$ such that $x^2+(kx)^2 = (1+k^2)x^2 = r$. Hence, $\sharp(E\cap S_r)=1$. This completes the proof for $d=2$. 

  \medskip

  When $d>2$, we take $E$ in ${\mathbb F}_q^2$ be the set we just constructed above and define $\widetilde{E} = E\times{\mathbb F}_q^{d-2}$. Then $\widetilde E$   is   a multiplicative tiling set on ${\mathbb F}_q^d$ by Theorem \ref{theorem_multiplicative}. Moreover, $\widetilde{E}$ is also a translational tile  with respect to a coordinate system. For example if we choose   ${\bf e}_2 = (0,1,0,\cdots,0)$, the tiling set we obtain is  $\{r{\bf e}_2: \ 0\leq r\leq q-1\}$. 
  To prove that $\widetilde E$ has a spectrum, let $L$ be a spectrum for $E$. A simple calculation shows that $L\times {\mathbb F}_q^{d-2}$ is a spectrum for $\widetilde{E}$, and this completes the proof of the theorem.   
\end{proof}

{\it Remark.} Notice in Theorem \ref{Thm:admissible sets}  the  tiling and spectral sets for $d>2$ are given by $\widetilde{E} = E\times{\mathbb F}_q^{d-2}$. However, other natural candidate is    $\widetilde{E} := E\times E\times \cdots \times E  \times {\mathbb F}_q^k$ where $0\leq k\leq d-2$. Clearly,   the assertions of the  theorem holds for $\widetilde E$.  

\medskip
As a corollary of  Theorem \ref{Thm:admissible sets} and Theorem \ref{PWF} we have the following result. 
 
\begin{corollary} Let  $q$ be an odd prime congruent to $3$ (mod 4) and $d\ge 2$. Let $Y:=  (\Bbb F_q^d)^*$. Then there exists   tight wavelet frame sets  for $PW_Y$. 
\end{corollary}

\medskip

Our result settles the existence of wavelet sets when $q$ is an odd prime and $q\equiv3$ (mod 4). The following proposition shows however that our construction method cannot work for $q\equiv1$ (mod 4). The reason is that the circle of radius $0$, $S_0$, contains more than one point.

\medskip
\begin{proposition}
If $q\equiv1$(mod 4), there cannot be wavelet sets obtained by rotational tilings and translations in $\Bbb F_q^2$. 
\end{proposition}

\begin{proof} Take $Y:=\Bbb F_q^2\setminus \{0\}$. 
If $E$ is a non-trivial set in ${\mathbb F}_q^2$ which tiles  $Y$ by a set of rotations and translations, then $\# E =q$ and $E$ is a graph. Also, $\sharp(E\setminus\{0\}) = q-1$ and it contains exactly one point from each circle. However, there are $q-1$ circles of positive radius. Taking one point from each circle of positive radius would occupy all points in $E\setminus\{0\}$. Taking union of all possible rotations, the rotational tiling   covers only points of non-zero radius. Therefore the points in the circle with  zero radius $S_0$ are not covered. Note that for $q\equiv1$ (mod 4),   $\sharp S_0=2q-1$.  Therefore there  cannot be rotational tiling sets thus wavelet sets obtained by rotation and translation when  $q\equiv1$ (mod 4).
\end{proof}
%
%
%
\section{Open problems} 
%
%
We end this paper with two open   questions. 

{\bf Question $\#1$:} Does there exist tight frame wavelet sets when $q\equiv1$ (mod 4)?

To handle the case $q \equiv 3 (mod 4)$ we exploited the fact that circle of $0$ radius has only one element. We would have to come to grips with such circles to extend our results to the case $q \equiv 1 (mod 4)$. 

{\bf Question $\#2$:} To what extent is it possible to generalize the results of this paper to the case $\Bbb F_{q}^d$, for $q=p^\alpha$,  $\alpha>1$? 

By restricting ourselves to the case where $q$ is prime, congruent to $1 \mod 4$, we limited the impact of arithmetic intricacies on the problem. The situation becomes quite fascinating when $\alpha>1$ due to the existence of subfields. We shall address this issue in a sequel.  \\ 

 
 {\bf Acknowledgement} 
The second author was supported by Minigrant (Grant NO: ST659) of ORSP of San Francisco State University. The authors would like to thank the Graduate Center of the City University of New York where this project was initiated. 


\begin{thebibliography}{10}

\bibitem{Baggett99}
L. Baggett, H. Medina, and K. Merrill, {\it Generalized multiresolution analyses and a construction
procedure for all wavelet sets in $\Bbb R^n$}, J. Fourier Anal. Appl. 5 (1999), 563--573


\bibitem{Benedetto-Benedetto04} J. J. Benedetto, R. L.  Benedetto, {\it A theory for local fields and related groups}, The Journal of Geometric Analysis, Vol. 14, No. 3, 2004.

\bibitem{Benedetto-Benedetto11} J. J. Benedetto, R. L.  Benedetto,
 {\it The Construction of Wavelet Sets}, Wavelets and Multiscale Analysis,
  Applied and Numerical Harmonic Analysis,  pp 17-56, 2011

\bibitem{Benedetto-Leon99} J. Benedetto and M. Leon, {\it The construction of multiple dyadic minimally supported frequency
wavelets on $\Bbb R^n$}, Contemp. Math. 247 (1999), 43--74

\bibitem{Larson-Speegle-Dai} X. Dai,  D. R. Larson, D. M. Speegle,  {\it Wavelet sets in $\Bbb R^n$}, The Journal of Fourier Analysis and Applications 3 (1997), no. 4, 451--456.

\bibitem{Duffin-Schaeffer} R. J. Duffin, A. C. Schaeffer, {\it A class of nonharmonic Fourier series},
Trans. Amer. Math. Soc. 72 (1952), 341--366.


\bibitem{DL}
 D. Dutkay and C.-K. Lai, {\it Some reductions of the spectral set conjecture to integers}, Math Proc Cambridge Phil Soc., 156 (2014), 123-135.


\bibitem {F}
 B. Fuglede, {\it Commuting self-adjoint partial differential operators and a group theoretic problem},
\ J. Funct. Anal. {16}(1974), 101-121


\bibitem{Han-Larson} D. Han, D. Larson, {\it Frames, Bases and Group Representations} ,
Memoirs of the American Mathematical Society, 2000.

\bibitem{HWW1}
E. Hernandez, X.-H Wang, G. Weiss, {\it Smoothing minimally supported frequency wavelets}, I, J. Fourier Analysis and Applications, 2 (1996), 329--340.

\bibitem{HWW2}
E. Hernandez, X.-H Wang, G. Weiss, {\it Smoothing minimally supported frequency wavelets}, II, J. Fourier Analysis and Applications, 3 (1997), 23--4.

\bibitem{ILMP16} A. Iosevich, A. Liu, A. Mayeli and J. Pakianathan {\it Wavelet decomposition and bandwidth of functions defined on vector spaces over finite fields}, (submitted for publication), (https://arxiv.org/pdf/1601.03473.pdf) (2016). 

\bibitem{IM_FA} A. Iosevich, A. Mayeli, {\it Exponential bases, Paley-Wiener spaces, and applications}; Journal of Functional Analysis, Volume 268, Issue 2, 15 January 2015, Pages 363--375.

\bibitem{IMP15} A. Iosevich, A. Mayeli, J.  Pakianathan, {\it The Fuglede Conjecture holds in ${\Bbb Z}_p \times {\Bbb Z}_p$}, to appear in Analysis and PDE. (https://arxiv.org/pdf/1505.00883.pdf)

\bibitem{Merrill08}  K. D. Merrill, {\it Simple wavelet sets for scalar dilations in $L^2(\Bbb R^2)$},  in Wavelets and Frames: A
Celebration of the Mathematical Work of Lawrence Bagget (P. Jorgensen, K. Merrill, and J.
Packer, eds.), Birkhaeuser, Boston, 2008, 177--192

 \bibitem{Merrill12}  K. D. Merrill, {\it Simple wavelet sets for matrix dilations in $\Bbb R^2$}, Numer. Funct. Anal. Optim.
33(7-9) (2012), 1112-1125.


\bibitem{ME} C. Monico, M. Elia, {\it Notes on an Additive Characterization of Quadratic Residues Modulo p}, Journal of Combinatorics, Information and System Sciences, 31 (2006), 209--215.

 \bibitem {T}  T. Tao, {\it Fuglede's conjecture is false in 5 or higher dimensions}, Math. Res. Letter, 11(2004), 251-258.


\bibitem{Fang-Wang} X. Fang, X. Wang, {\it
Construction of Minimally-Supported-Frequencies Wavelets}, the Journal of Fourier Analysis and Application 2 (1996), no. 4, 315-327.

\bibitem{Wang} Y. Wang, {\it Wavelets, tiling, and spectral sets},  Duke Math. J. 114 (2002), no. 1, 43--57.




\end{thebibliography}
\end{document}